\documentclass[12pt]{amsart}
\usepackage{amsmath, amsthm, amscd, amsfonts, amssymb, graphicx, xcolor}
\usepackage{hyperref}

\usepackage{mathtools}
\usepackage{enumitem} 

\usepackage{tikz-cd} 

\usepackage{graphicx}
\graphicspath{ {images/} }

\newcommand{\N}{\ensuremath{\mathbb{N}}}

\newcommand{\C}{\ensuremath{\mathbb{C}}}
\newcommand{\D}{\ensuremath{\mathbb{D}}}
\newcommand{\Sone}{\ensuremath{\mathbb{S}^{1}}}

\newcommand{\eps}{\ensuremath{\varepsilon}}

\DeclareMathOperator\capacity{Cap}

\newtheorem{theorem}{Theorem}[section]
\newtheorem{lemma}[theorem]{Lemma}
\newtheorem{proposition}[theorem]{Proposition}
\newtheorem{corollary}[theorem]{Corollary}
\theoremstyle{definition}
\newtheorem{definition}[theorem]{Definition}

\theoremstyle{remark}
\newtheorem{remark}[theorem]{Remark}
\numberwithin{equation}{section}

\begin{document}
\setcounter{page}{1}

\color{darkgray}{
\noindent 

\centerline{}

\centerline{}

\title[Every circle homeomorphism is the composition of two weldings]{Every circle homeomorphism is the composition of two weldings}

\author[Alex Rodriguez]{Alex Rodriguez}
\address{Department of Mathematics, Stony Brook University, New York, USA.\\
	\textsc{\newline \indent 
	   \href{https://orcid.org/0000-0001-9097-4025%
	     }{\includegraphics[width=1em,height=1em]{orcid2} {\normalfont https://orcid.org/0000-0001-9097-4025}}
	       }}}
\email{\textcolor[rgb]{0.00,0.00,0.84}{alex.rodriguez@stonybrook.edu}}



\date{January 17, 2025; revised February 5, 2025.
\newline \indent The author is partially supported by NSF grant DMS 2303987 and the Simons Foundation} 

\begin{abstract}
We show that every orientation-preserving circle homeomorphism is a composition of two conformal welding homeomorphisms, which implies that conformal welding homeomorphisms are not closed under composition. Our approach uses the log-singular maps introduced by Bishop. The main tool that we introduce are log-singular sets, which are zero capacity sets that admit a log-singular map that maps their complement to a zero capacity set.
\end{abstract} 

\maketitle


\section{Introduction}

Let $\gamma\subset\C_{\infty}$ be a Jordan curve and consider $\Omega, \Omega^{*}$ the two complementary components of $\gamma$ in $\C_{\infty}$, the Riemann sphere. By the Riemann mapping theorem there are conformal maps $f\colon\D\to\Omega$ and $g\colon\C_{\infty}\setminus\overline{\D}=\D^{*}\to\Omega^{*}$, as represented in Figure \ref{figure:welding_def}. By the Carath\'eodory-Torhorst theorem, these conformal maps extend to be homeomorphisms of the unit circle to $\gamma$. This yields the orientation-preserving homeomorphism $h=g^{-1}\circ f\colon\Sone\to\Sone$ which we call a \textit{conformal welding} (or \textit{welding} for short).

It is well known that there are orientation-preserving circle homeomorphisms that are not weldings. For instance, Oikawa provided in \cite[Example 1]{MR125956} a family of such examples that are biH\"older. However, well-behaved self-maps of the circle, like quasisymmetric maps, are weldings. This is the so called \textit{fundamental theorem of conformal welding} and it was first proven by Pfluger in \cite{Pfluger} by using the measurable Riemann mapping theorem. Shortly after, Lehto and Virtanen \cite{MR125962} gave a different proof, also by using quasiconformal mappings. See Hamilton \cite{Hamilton:WeldingSurvey} for a survey on conformal welding.

In this paper we prove the following.

\begin{theorem}\label{theorem:mainWelding}
Every orientation-preserving circle homeomorphism is the composition of two conformal weldings.
\end{theorem}

In particular,

\begin{corollary}\label{corollary:composition}
The set of conformal welding homeomorphisms is not closed under composition.
\end{corollary}

The first proof of Corollary \ref{corollary:composition} appeared in Vainio's thesis \cite[Section 2.2]{Vainio:Welding}. Later in the introduction, we will show how this follows from some previous work of Bishop \cite{ChrisWeldingAnnals} and the examples provided by Oikawa \cite{MR125956}. Another example was recently given by McMullen \cite{McMullen:Welding}.

\begin{figure}[h]
\includegraphics[scale=1]{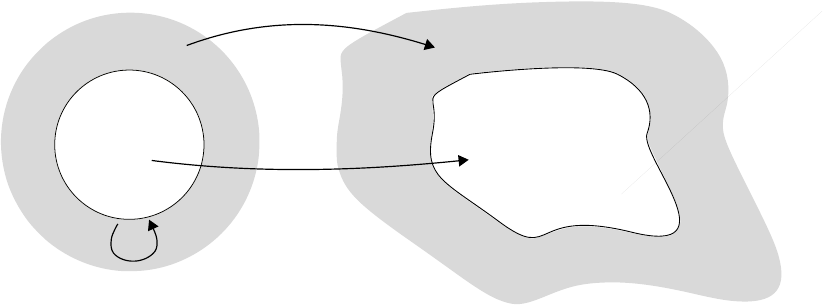}
\centering
\setlength{\unitlength}{\textwidth}
\put(-0.62,0.172){$f$}
\put(-0.62,0.31){$g$}
\put(-0.19,0.26){$\gamma$}
\put(-0.9,0.26){$\Sone$}
\put(-0.772,0.059){$h=g^{-1}\circ f$}
\caption{Definition of conformal welding. The conformal maps $f$ and $g$ map the interior and exterior of the unit circle to the two complementary components of a Jordan curve $\gamma$. The map $h\colon\Sone\to\Sone$ is an orientation-preserving homeomorphism of $\Sone$.}
\label{figure:welding_def}
\end{figure}
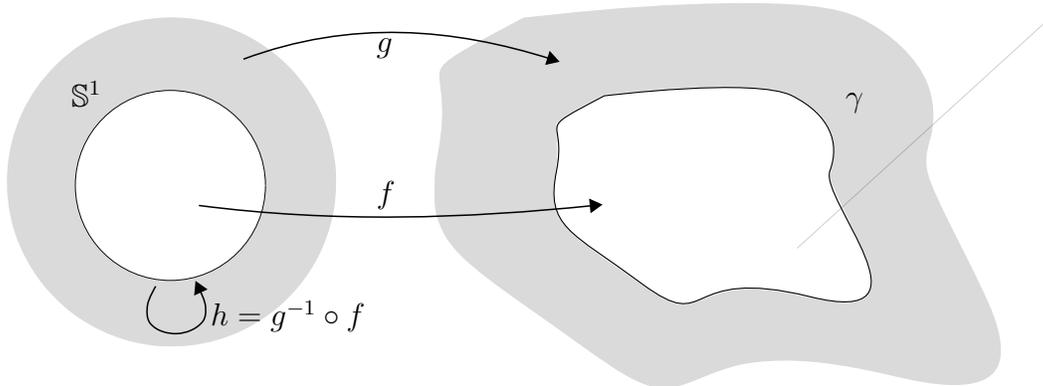

There have been some important breakthroughs regarding conformal weldings, like Bishop's paper \cite{ChrisWeldingAnnals}, in which he proves that a very \textit{badly} behaved class of circle homeomorphisms are weldings. Some work from Astala, Jones, Kupiainen and Saksman \cite{JonesActa2011} proves that regularity \textit{at most scales} is a sufficient condition for a homeomorphism to be a welding (which is satisfied for quasisymmetric maps). However, many important questions are still open, like showing whether or not the collection of weldings is a Borel subset of the space of circle homeomorphisms. In \cite[Section 9]{Chris:Borel} Bishop proves that it is at least analytic. He also raises several questions regarding weldings and their relations to other problems in geometric function theory.

Conformal welding has been used in many other areas of mathematics. For example, Sheffield \cite{Sheffield:Welding} and Duplantier, Miller and Sheffield \cite{Sheffield:LQG} used conformal welding to relate Liouville quantum gravity and Schramm-Loewner evolution curves. See also Ang, Holden and Sun \cite{Ang:SLE_Welding}, where they prove that SLE$_{\kappa}$ measure, for $\kappa\in(0,4)$, arises naturally from the conformal welding of two $\sqrt{\kappa}$-Liouville quantum gravity disks. Sharon and Mumford \cite{Mumford:Welding} used conformal welding in their applications to computer vision. Conformal welding also lies at the basis of the construction of universal Teichm\"uller space (see for instance Lehto's book \cite{Lehto:Teichmuller}).

What maps can be used to prove Theorem \ref{theorem:mainWelding}? Well-behaved classes of homeomorphisms, like diffeomorphisms, biLipschitz maps, quasisymmetric maps and biH\"older maps, cannot be used since they each form a proper subgroup of the group of circle homeomorphisms. However, there are circle homeomorphisms with arbitrarily bad modulus of continuity.

We say that a circle homeomorphism $h\colon\Sone\to\Sone$ is \textit{log-singular} if there is a (logarithmic) zero capacity Borel set $E\subset\Sone$ so that $h(\Sone\setminus E)$ has zero capacity. Log-singular maps are stronger versions of the better known singular maps $\phi$ so that $\phi'=0$ a.e., since they map a set of full Lebesgue measure to a measure zero set.
One of the main results obtained in \cite{ChrisWeldingAnnals} is the following.

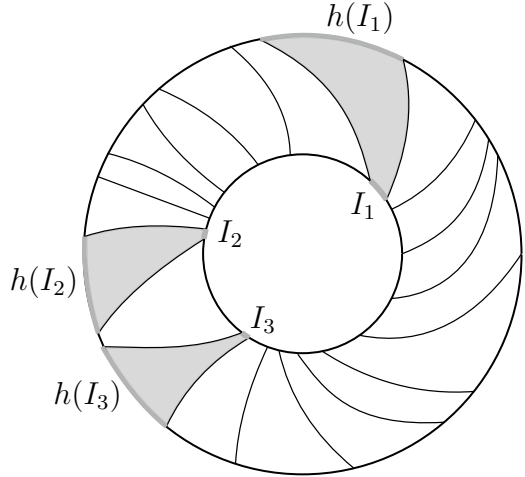
\begin{figure}[h]
\includegraphics[scale=1]{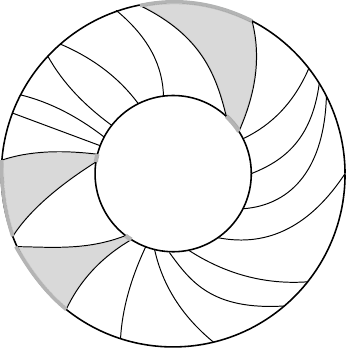}
\centering
\setlength{\unitlength}{\textwidth}
\put(-0.43,0.065){$h(I_{3})$}
\put(-0.25,0.135){$I_{3}$}
\put(-0.47,0.17){$h(I_{2})$}
\put(-0.28,0.214){$I_{2}$}
\put(-0.18,0.412){$h(I_{1})$}
\put(-0.158,0.24){$I_{1}$}
\caption{Representation of a log-singular circle homeomorphism $h\colon\Sone\to\Sone$. The arcs $I_{j}$, which have small capacity, are mapped to the arcs $h(I_{j})$, which have larger capacity.}
\label{figure:logsingular_ex}
\end{figure}

\begin{theorem}[\cite{ChrisWeldingAnnals} Theorem 3]\label{Theorem:ChrisLogSingular}
Every log-singular circle homeomorphism $h\colon\Sone\to\Sone$ is a conformal welding homeomorphism.
\end{theorem}

In light of Theorem \ref{Theorem:ChrisLogSingular}, we see that Theorem \ref{theorem:mainWelding} reduces to proving the following.
 
\begin{theorem}\label{theorem:mainWelding2}
Let $\phi\colon\Sone\to\Sone$ be a circle homeomorphism, then there exists a log-singular map $h\colon\Sone\to\Sone$ so that $\phi\circ h^{-1}$ is also log-singular. Thus $\phi = (\phi \circ h^{-1}) \circ  h$ is a composition of two conformal weldings.
\end{theorem}

Observe that Theorem \ref{theorem:mainWelding2} follows immediately from Theorem \ref{Theorem:ChrisLogSingular} if the circle homeomorphism $\phi$ has enough regularity. For instance, biH\"older homeomorphisms preserve zero capacity sets. Therefore, if $\phi\colon\Sone\to\Sone$ is biHolder, then $\phi\circ h^{-1}$ is log-singular for any log-singular circle homeomorphism $h\colon\Sone\to\Sone$. Since $\phi=(\phi\circ h^{-1})\circ h$, and $\phi\circ h^{-1}$ and $h$ are weldings by Theorem \ref{Theorem:ChrisLogSingular}, every biH\"older circle homeomorphism is a composition of two weldings. The examples of non-welding homeomorphisms provided by Oikawa \cite[Example 1]{MR125956} are biH\"older, hence weldings are not closed under composition (Corollary \ref{corollary:composition}).
Similarly, any circle homeomorphism that preserves sets of zero capacity is also a composition of two (log-singular) weldings. The novel aspect of this paper is to prove that this holds for all (orientation-preserving) circle homeomorphisms.

Given any circle homeomorphism $\phi$, it suffices to build a log-singular homeomorphism $h$ and a set $E$ of zero capacity so that both $h(\Sone \setminus E)$ and $\phi(E)$ have zero capacity, for then $\phi\circ h^{-1}$ and $h$ are log-singular.  The set $E$ that we will construct will be a countable union of Cantor sets, and it will be described using an infinite tree whose vertices are labeled by finite strings of integers (which we call \textit{addresses}), much like a binary tree represents dyadic intervals. In our case, the number of children of a vertex increases in each generation, and, unlike the dyadic case, intervals in the same generation may have very different sizes.

By using addresses we provide a large class of zero capacity sets $E$, so that there exists a log-singular homeomorphism $h\colon\Sone\to\Sone$ with $h(\Sone\setminus E)$ of zero capacity, which we call a \textit{log-singular set}. In fact, any log-singular map admits as a log-singular set one of the zero capacity sets that we construct.

The paper is structured in the following way.

	\textbf{Section \ref{section:log_capacity}:} We introduce logarithmic capacity, quasisymmetric maps and flexible curves.
	
	\textbf{Section \ref{section:logsingular}:} We introduce log-singular sets and addresses, which encode partitions of a given interval or arc. We use addresses to build log-singular sets and log-singular maps.
	
	\textbf{Section \ref{section:finding}:} We prove that given a circle homeomorphism $\phi$ there is a log-singular set $E\subset\Sone$, given by an address, so that $\capacity(\phi(E))=0$. We also prove Theorem \ref{theorem:mainWelding2}.

 \subsection*{Acknowledgements} 

I would like to thank Chris Bishop for his continued guidance and support, and for reading (many) drafts of the paper; his comments, suggestions and remarks highly improved the presentation. I would like to thank  Mario Bonk, Steffen Rohde, Dennis Sullivan and Yilin Wang for their suggestions and encouraging conversations about the result. Malik Younsi informed me that Corollary \ref{corollary:composition} was first proved by Vainio. Spencer Cattalani has found many typos and grammatical mistakes.


\section{Preliminaries}\label{section:log_capacity}

In this section we summarize known results regarding (logarithmic) capacity that we will use in the following sections. These results are all contained in the books of Carleson, Garnett and Marshall, and Pommerenke \cite{CarlesonBook, HarmonicMeasure, PommerenkeBook}. We also recall the definition of quasisymmetric homeomorphisms and provide a summary of the main results obtained by Bishop in \cite{ChrisWeldingAnnals}.

\subsection{Logarithmic capacity}

Let $\mu$ be a finite compactly supported signed (Borel) measure. The \textit{logarithmic potential} of $\mu$ is the functional $$ U_{\mu}(z)=\int \log\frac{1}{|\xi-z|}d\mu(\xi).$$

By Fubini's theorem, $U_{\mu}(z)$ is finite a.e. with respect to Lebesgue measure. If $$ \iint\left| \log\frac{1}{|z-\xi|}\right| d|\mu|(\xi)d|\mu|(z)<\infty,$$ we say that $\mu$ has \textit{finite energy} and define the \textit{energy integral} $I(\mu)$ by
$$ I(\mu)=\iint \log\frac{1}{|z-\xi|} d\mu(\xi)d\mu(z)=\int U_{\mu}(z)d\mu(z).$$

Let $K$ be compact and denote by $P(K)$ the set of all Borel probability measures on $K$. Define the \textit{Robin's constant} of $K$ by 
$$ \gamma(K)=\inf\{I(\mu)\colon \mu\in P(K)\},$$
and the (logarithmic) \textit{capacity} of $K$ by 
$$ \capacity(K)=e^{-\gamma(K)}.$$
It follows that $\gamma(K)<\infty$ if and only if $\capacity(K)>0$, which holds as long as there is a probability $\mu\in P(K)$ with finite energy $I(\mu)<\infty$. If $K_{1}\subset K_{2}$, then any probability on $K_{1}$ is also a probability on $K_{2}$. Hence $\gamma(K_{1})\geq\gamma(K_{2})$ and $\capacity(K_{1})\leq\capacity(K_{2})$, i.e. the capacity set function is monotone with respect to set inclusion.

The \textit{capacity} of any Borel set $E$ is defined as 
$$\capacity(E)=\sup\{\capacity(K)\colon K \textrm{ compact}, K\subset E\}.$$

Capacity can be a difficult object to compute, but the capacities of some simple sets are know (see \cite[Chapter 3]{HarmonicMeasure} or \cite[Chapter 9]{PommerenkeBook}). For example, for $a,b\in\C$, $\capacity[a,b]=|a-b|/4$, and the capacity of a disk is its radius, i.e. $\capacity(\D(a,r))=r$.

We now summarize some other known properties of logarithmic capacity that we will need later.

\begin{proposition}\label{proposition:capacity_properties} Logarithmic capacity satisfies the following properties.
\begin{enumerate}
	\item If $\varphi$ is $L$-Lipschitz, then $\capacity(\varphi(E))\leq L\capacity(E)$.
	\item If $h(\xi)=c\xi+c_{0}+c_{1}\xi^{-1}+\cdots$ maps $\C_{\infty}\setminus\overline{\D}$ conformally onto $G$ then $\capacity(\partial G)=|c|$.
	\item Borel sets are capacitable, that is, given $E$ Borel and $\eps>0$, there is an open set $V$ with $E\subset V$ and $\capacity(V)<\capacity(E)+\eps$.
	\item If $E_{n}$ are Borel sets so that $E=\cup E_{n}$ has diameter less or equal than one, then $\capacity(E)\leq\sum\capacity(E_{n})$.
\end{enumerate}
\end{proposition}

The proof of (1), (2) and (3) can be found in \cite[Chapter 9]{PommerenkeBook},  (4) follows from Lemma 4 in \cite[Section 3]{CarlesonBook}.

We say that a circle homeomorphism $\phi\colon\Sone\to\Sone$ is $\alpha$-biHolder if there is a constant $C>0$ so that $$ \frac{1}{C}|z-w|^{1/\alpha}\leq|\phi(z)-\phi(w)|\leq C|z-w|^{\alpha}.$$

As it was already mentioned in the introduction, biHolder maps preserve sets of zero capacity.

\begin{lemma}
Let $\phi$ be $\alpha$-biHolder, then $\capacity(E)=0$ if and only if $\capacity(\phi(E))=0$.
\end{lemma}

This follows by making a minor alteration in the proof of \cite[Chapter 3, Theorem 4.5]{HarmonicMeasure}.

\subsection{Quasisymmetric homeomorphisms}

We say that a circle homeomorphism $\phi\colon\Sone\to\Sone$ is quasisymmetric if there is a constant $M<\infty$ so that $$ M^{-1}\leq |\phi(I)|/|\phi(J)|\leq M,$$ whenever $I,J\subset\Sone$ are two adjacent arcs of equal length (see \cite{AhlforsQC,AstalaBook,LehtoVirtanen} for more information regarding quasisymmetric homeomorphisms).

Pfluger proved that quasisymmetric homeomorphisms are weldings.

\begin{theorem}[Pfluger \cite{Pfluger}]
Every quasisymmetric circle homeomorphism is a conformal welding.
\end{theorem}

By \cite[Corollary 3.10.3]{AstalaBook}, every quasisymmetric circle homeomorphism is $\alpha$-biHolder for some $\alpha>0$. Therefore quasisymmetric circle homeomorphisms preserve sets of zero capacity.

\subsection{Flexible curves and log-singular homeomorphisms}

In \cite{MR1274085} Bishop introduced and proved the existence of flexible curves. We say that a Jordan curve $\gamma\subset\C$ is \textit{flexible} if the following conditions are satisfied:
\begin{enumerate}[label=(\alph*)]
	\item Given any other Jordan curve $\tilde{\gamma}$ and any $\eps>0$, there exists a homeomorphism $\phi\colon\C_{\infty}\to\C_{\infty}$, conformal off $\gamma$, so that the Hausdorff distance between $\phi(\gamma)$ and $\tilde{\gamma}$ is less than $\epsilon$.
	\item Given $z_{1},z_{2}$ in each component of $\C_{\infty}\setminus\gamma$ and $w_{1},w_{2}$ in each component of $\C_\infty\setminus\tilde{\gamma}$, the previous $\phi$ can be taken so that $\phi(z_{j})=w_{j}$ for $j=1,2$.
\end{enumerate}

Flexible curves are conformally non-removable in a very strong sense. Later in \cite{ChrisWeldingAnnals} Bishop proved that $h\colon\Sone\to\Sone$ is the conformal welding of a flexible curve if and only if $h$ is log-singular. The specific result that he proved is the following, which implies Theorem \ref{Theorem:ChrisLogSingular}.

\begin{theorem}[\cite{ChrisWeldingAnnals} Theorem 3]
Let $h\colon\Sone\to\Sone$ be a log-singular homeomorphism with $h(1)=1$, and let $f_{0}$ and $g_{0}$ be two conformal maps of $\D$ and $\C_{\infty}\setminus\overline{\D}$ respectively onto disjoint domains. Then, for any $0<r<1$ and any $\eps>0$, there are conformal maps $f$ and $g$ of $\D$ and $\C_{\infty}\setminus\overline{\D}$ onto the two complementary components of a Jordan curve $\gamma$ that satisfy:
\begin{enumerate}[label=(\roman*)]
	\item $h=g^{-1}\circ f$ on $\Sone$, i.e. $h$ is a conformal welding.
	\item $|f(z)-f_{0}(z)|<\eps$ for all $|z|\leq r$.
	\item $|g(z)-g_{0}(z)|<\eps$ for all $|z|\geq 1/r$.
\end{enumerate}
Moreover, the curve $\gamma$ can be taken to have zero area.
\end{theorem}

In \cite{Burkart:JordanCurve} Burkart constructs a flexible curve that cannot be crossed by a rectifiable arc on a set of zero length. In \cite{PughWu} Pugh and Wu construct some special Jordan curves that arise as funnel sections of some ODEs. Those curves are flexible, as it is explained in \cite{Burkart:JordanCurve}.

Theorem \ref{theorem:mainWelding} (and Theorem \ref{theorem:mainWelding2}), that we prove in this paper, shows that any circle homeomorphism can be written as a composition of two weldings. In \cite{ChrisWeldingAnnals} Bishop proved that every circle homeomorphism is a welding on a set of nearly full measure.

\begin{theorem}[\cite{ChrisWeldingAnnals} Theorem 1]
Given any orientation-preserving homeomorphism $\phi\colon\Sone\to\Sone$ and $\eps>0$, there are a set $E\subset\Sone$ with $|E|+|h(E)|<\eps$ and a conformal welding homeomorphism $h\colon\Sone\to\Sone$ such that $\phi(x)=h(x)$ for all $x\in\Sone\setminus E$.
\end{theorem}


\section{Log-singular sets and addresses}\label{section:logsingular}

In this section, we give sufficient conditions for a zero capacity set $E$ to admit a log-singular homeomorphism $h$ satisfying $\capacity(h(\Sone\setminus E))=0$, which we call a log-singular set. The approach is based on what we call \textit{addresses}, which encode partitions of a fixed arc of $\Sone$ (or an interval) by sub-arcs (or sub-intervals) in an iterative way.

We proceed to define addresses. It suffices to do so for $I=[0,1]$, since the definition will work for any other interval or sub-arc of $\Sone$ in an analogous way.

Let $\{L_{n}\}_{n\geq0}$ be a sequence of positive integers with $L_{0}=1$. We say that $A_{n}=\{a_{1}, a_{2}, \ldots, a_{n}\}$, where $0\leq a_{j}<L_{j}$ for $1\leq j\leq n$, is a word of length $n$. Suppose that for $n\geq1$ we have have intervals $I_{A_{n}}\subset[0,1]$ that satisfy:
\begin{enumerate}[label=(\roman*)]
	\item If $n\leq m$ and $A_{n}=\{a_{1}, a_{2}, \ldots, a_{n}\}$, $\tilde{A}_{m}=\{\tilde{a}_{1},\ldots, \tilde{a}_{n},\ldots,\tilde{a}_{m}\}$, then $I_{\tilde{A}_{m}}\subset I_{A_{n}}$ if and only if $a_{j}=\tilde{a}_{j}$ for $1\leq j\leq n$ (in such a case, $A_{n}$ is often called the prefix of $\tilde{A}_{m}$).
	\item If $A_{n}\not=\tilde{A}_{n}$, then $I_{A_{n}}$ and $I_{\tilde{A}_{n}}$ have disjoint interiors.
	\item $L_{n}$ denotes the number of intervals in which each $I_{A_{n-1}}$ has been divided, which we assume is the same for all $I_{A_{n-1}}$.
	\item For each $n$, the union over all possible $I_{A_{n}}$ is $I$. Moreover, if we fix a word $A_{n-1}=\{a_{1}, \ldots, a_{n-1}\}$, then the union over all possible $A_{n}=\{a_{1},\ldots, a_{n-1}, a\}$, with $0\leq a< L_{n}$, is $I_{A_{n-1}}$, i.e. $$ I_{A_{n-1}}=\bigcup_{0\leq a<L_{n}} I_{\displaystyle\{a_{1},\ldots, a_{n-1}, a\}}.$$
\end{enumerate}

We will say that $A_{n}$ is the \textit{address} of $I_{A_{n}}$; $A_{n}$ represents the directions that we need to take in the decomposition to reach $I_{A_{n}}$, in the sense that if we fix $A_{n}=\{a_{1},\ldots,a_{n}\}$, then 
\begin{equation}\label{equation:1} 
I_{A_{n}}=\bigcap_{j=1}^{n}I_{\{a_{1},\ldots, a_{j}\}}.
\end{equation}

\begin{figure}[h] 
\begin{center}
\begin{tikzcd}
                    &                                                   & I \arrow[d,thick] \arrow[lld,dashed] \arrow[rrd,dashed]                               &                                                           &                 \\
I_{\{0\}}               &                                                   & I_{\{1\}} \arrow[rrd,dashed] \arrow[rd,thick] \arrow[d,dashed] \arrow[ld,dashed] \arrow[lld,dashed] &                                                           & I_{\{2\}}           \\
{I_{\{1,0\}}}       & {I_{\{1,1\}}}                                     & {I_{\{1,2\}}}                                                     & {I_{\{1,3\}}} \arrow[lld,thick] \arrow[ld,dashed] \arrow[d,dashed] \arrow[rd,dashed] & {I_{\{1,4\}}}   \\
                    & {I_{\{1,3,0\}}} \arrow[ld,dashed] \arrow[d,thick]              & {I_{\{1,3,1\}}}                                                   & {I_{\{1,3,2\}}}                                           & {I_{\{1,3,3\}}} \\
{I_{\{1,3,0,0\}}}   & {I_{\{1,3,0,1\}}} \arrow[ld,dashed] \arrow[d,dashed] \arrow[rd,thick] &                                                                   &                                                           &                 \\
{I_{\{1,3,0,1,0\}}} & {I_{\{1,3,0,1,1\}}}                               & {I_{\{1,3,0,1,2\}}}                                               &                                                           &                
\end{tikzcd}
\end{center}
\caption{Addresses associated to a partition of an interval $I$ with $L_{1}=3, L_{2}=5, L_{3}=4, L_{4}=2, L_{5}=3$. The address to $I_{\{1,3,0,1,2\}}$ is highlighted in a thicker color and by non-dashed arrows.}
\label{figure:explanation_addresses}
\end{figure}
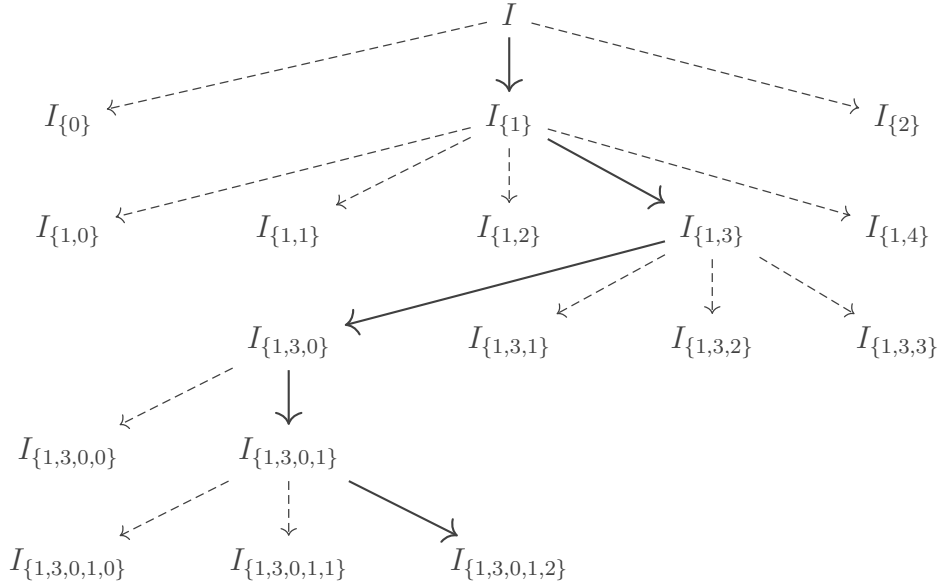

This is illustrated in Figure \ref{figure:explanation_addresses}. If $A_{n-1}=\{a_{1},\ldots, a_{n-1}\}$ we will also set, for $0\leq a< L_{n}$, 
\begin{equation}\label{equation:2}
A_{n-1}(a)=\{a_{1},\ldots, a_{n-1}, a\}
\end{equation} 
to denote that we extend the word $A_{n-1}$ by $a$.

Similarly, given an infinite sequence $\{a_{j}\}_{j\in\N}$, with $0\leq a_{j}<L_{j}$ for each $j\in\N$, we denote $A_{\infty}=\{a_{1},a_{2},\ldots\}$. Observe that given $A_{\infty}=\{a_{1},a_{2},\ldots\}$, then 
\begin{equation}\label{equation:3}
I_{A_{\infty}}=\bigcap_{n\geq1} I_{\{a_{1},\ldots,a_{n}\}}
\end{equation}
is either an interval or a point (since the lengths of the intervals decrease with $n$). If for any $\eps>0$, there exists $N\geq1$ so that for every $n\geq N$ we have $|I_{A_{n}}|<\eps$, then $I_{A_{\infty}}$ is a singleton. 

We will denote by $\mathcal{A}_{n}$ the collection of all possible words $A_{n}$ of length $n$, and by $\mathcal{A}_{\infty}$ the collection of all possible words $A_{\infty}$.

Well-known examples of such partitions include the usual binary or decimal representations of real numbers. For example, the dyadic partition of $[0,1]$ is given, in terms of addresses, by taking the sequence $\{L_{n}\}$ of positive integers with $L_{0}=1$ and $L_{n}=2$ for $n\geq1$. Each interval $I_{A_{n}}$ is split into two intervals of equal length (observe that the length of each $I_{A_{n}}$ is $2^{-n}$). We define the \textit{dyadic} partition of an interval $I=[a,b]$ as the partition given by the collection of intervals $\{I_{A_{n}}\}$ so that fixed $n$,
$$ I_{A_{n}}\subset \left\{ \left[a+\frac{j}{2^{n}},a+(b-a)\frac{j+1}{2^{n}}\right]\colon 0\leq j<2^{n} \right\}.$$

We recall the definition of log-singular homeomorphisms and log-singular sets.

\begin{definition}[Log-singular homeomorphism, log-singular set]
Let $I,J$ be two subarcs of $\Sone$. An orientation-preserving homeomorphism $h\colon I\to J$ is \textit{log-singular} if there exists a Borel set $E\subset I$ such that both $E$ and $h(I\setminus E)$ have zero logarithmic capacity. We say that a Borel set $E\subset\Sone$ is a \textit{log-singular set} if $\capacity(E)=0$ and there exists a log-singular homeomorphism $h\colon\Sone\to\Sone$ so that $\capacity(h(\Sone\setminus E))=0$.
\end{definition}

Observe that if $h\colon\Sone\to\Sone$ is a log-singular map with log-singular set $E\subset\Sone$, then $h(\Sone\setminus E)$ is also a log-singular set.

We can now prove that log-singular sets and log-singular maps exist. We will do so by using addresses, provided that we can obtain a dense capacity zero set.

\begin{lemma}[Log-singular sets and addresses]\label{lemma:logsingular}
Let $I,J$ be two subarcs of $\Sone$, and let $\{L_{n}\}$ be a sequence of positive integers so that $L_{0}=1$ and $L_{n}$, for $n\geq1$, is an even number. Consider words $A_{n}=\{a_{1},a_{2},\ldots, a_{n}\}$ with $0\leq a_{j}< L_{j}$. Suppose that $\{I_{A_{n}}\}$ is a partition of $I=I_{A_{0}}$ and $\{J_{A_{n}}\}$ is a partition of $J=J_{A_{0}}$ satisfying that given $\eps>0$, there exists $N\geq1$ such that for every $n\geq N$, $|I_{A_{n}}|+|J_{A_{n}}|<\eps$. Using the notation in (\ref{equation:1}) and (\ref{equation:2}), define $$ E=\bigcap_{m\in\N}\bigcup_{n\geq m} \left(\bigcup_{A_{n}\in\mathcal{A}_{n}}\bigcup_{j=1}^{L_{n+1}/2} I_{A_{n}(2j-1)}\right)$$ and $$ F=\bigcup_{m\in\N}\bigcap_{n\geq m}\left(\bigcup_{A_{n}\in\mathcal{A}_{n}}\bigcup_{j=1}^{L_{n+1}/2}J_{A_{n}(2j-2)}\right).$$
If $\capacity(E)=0$ and $\capacity(F)=0$, then $E$ and $F$ are log-singular sets and there exists a log-singular map $h\colon I\to J$ so that $h(I \setminus E)=F$.
\end{lemma}

\begin{remark}
In the conditions of Lemma \ref{lemma:logsingular}, given $A_{\infty}=\{a_{1},a_{2},\ldots\}\in\mathcal{A}_{\infty}$, then using the notation from (\ref{equation:3}), $$ I_{A_{\infty}}=\bigcap_{n\geq1}I_{\{a_{1},\ldots,a_{n}\}}=\{x_{A_{\infty}}\}$$ is a singleton. Hence we can define the map $\chi\colon\mathcal{A_{\infty}}\to[0,1]$ by $\chi(A_{\infty})=x_{A_{\infty}}$.

Therefore, the set $E$ in Lemma \ref{lemma:logsingular} corresponds to the image under $\chi$ of all possible $A_{\infty}=\{a_{1},a_{2},\ldots\}$, so that the sequence $\{a_{n}\}_{n}$ consists eventually of odd integers.
\end{remark}

\begin{proof}[Proof of Lemma \ref{lemma:logsingular}]
To prove the lemma we need to build a log-singular map $h\colon I\to J$ so that $h(\Sone\setminus E)=F$ . We will build $h$ in an iterative way by using the addresses $A_{n}$ and the intervals $I_{A_{n}}, J_{A_{n}}$ that are associated to them.

Take a linear bijective map $h_{0}\colon I\to J$. We can define a homeomorphism $h_{1}\colon I\to J$ that is linear on each $I_{\{j\}}$ and $h_{1}(I_{\{j\}})=J_{\{j\}}$ for $0\leq j<L_{1}$.

Suppose that at the $n$-th step, for $n>1$, we have a homeomorphism $h_{n}\colon I\to J$ which is linear on each $I_{A_{n}}$ and $h(I_{A_{n}})=J_{A_{n}}$. We define a homeomorphism $h_{n+1}\colon I\to J$ in the following way.
\begin{enumerate}[label=(\alph*)]
	\item $h_{n+1}$ is linear on each $I_{A_{n+1}}$.
	\item $h_{n+1}(I_{A_{j}})=J_{A_{j}}$ for every $j\leq n$.
\end{enumerate}
This procedure is illustrated in Figure \ref{figure:explanation_log-singular}. 

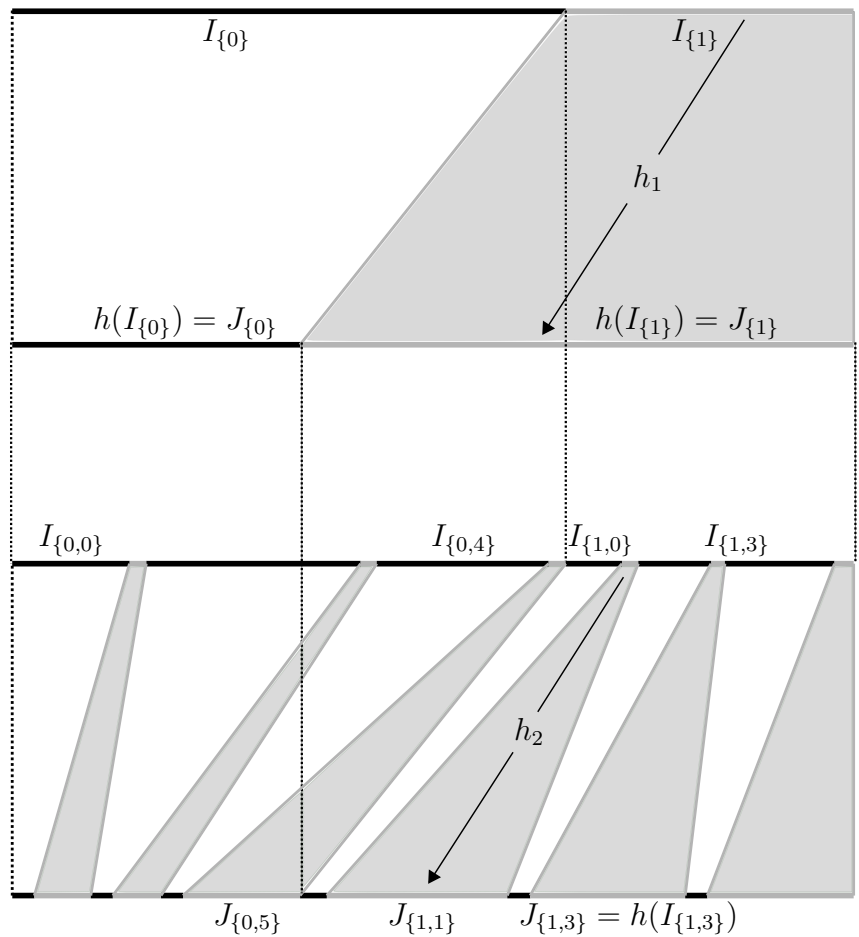
\begin{figure}[h]
\includegraphics[scale=1.5]{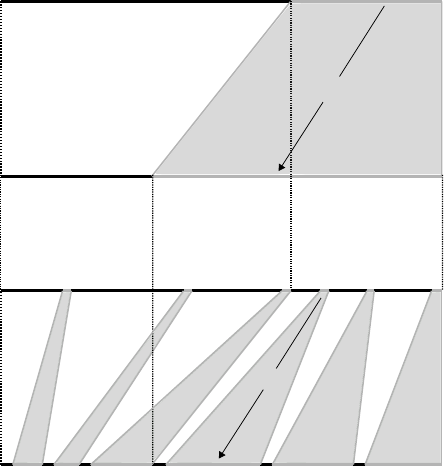}
\centering
\setlength{\unitlength}{\textwidth}
\put(-0.314,0.145){$h_{2}$}
\put(-0.14,0.323){$I_{\{1,3\}}$}
\put(-0.75,0.323){$I_{\{0,0\}}$}
\put(-0.39,0.323){$I_{\{0,4\}}$}
\put(-0.265,0.323){$I_{\{1,0\}}$}
\put(-0.205,0.653){$h_{1}$}
\put(-0.31,-0.024){$J_{\{1,3\}}=h(I_{\{1,3\}})$}
\put(-0.43,-0.024){$J_{\{1,1\}}$}
\put(-0.59,-0.024){$J_{\{0,5\}}$}
\put(-0.6,0.787){$I_{\{0\}}$}
\put(-0.17,0.787){$I_{\{1\}}$}
\put(-0.7,0.522){$h(I_{\{0\}})=J_{\{0\}}$}
\put(-0.24,0.522){$h(I_{\{1\}})=J_{\{1\}}$}
\caption{Sketch of the proof of Lemma \ref{lemma:logsingular} where two different partitions are \textit{matched up}. The figure represents two iterations of the process in the proof of Lemma \ref{lemma:logsingular}. The lighter intervals, like $I_{\{1,3\}}$, have small capacity in the domain and are mapped to intervals with larger capacity, like $J_{\{1,3\}}=h(I_{\{1,3\}})$. At each step $n$, the maps $h_{n}$ are linear on each $I_{A_{n}}$.}
\label{figure:explanation_log-singular}
\end{figure}

By construction they also satisfy that, if $n\leq m$, then $h_{m}(I_{A_{n}})=J_{A_{n}}$ and $h_{n}(I_{A_{m}})\subset J_{A_{m-1}}$. By hypothesis, given $\eps>0$ there is $N>0$ so that for $n\geq N$ the segments $J_{A_{n}}=h(I_{A_{n}})$ all have length less than $\eps$. Now, if $m\geq n$ and $x\in I$, there exists a word $A_{n}$ so that $x\in I_{A_{n}}$ and $h_{m}(x)\in J_{A_{n}}=h_{n}(I_{A_{n}})$. Hence, $$|h_{n}(x)-h_{m}(x)|<\eps.$$ Therefore the sequence $\{h_{n}\}$ is uniformly Cauchy and it has a continuous limit $h\colon I\to J$, which preserves orientation. 

If $z\not=w$, there are addresses $A_{\infty},\tilde{A}_{\infty}\in\mathcal{A}_{\infty}$ so that $\chi(A_{\infty})=z$ and $\chi(\tilde{A}_{\infty})=w$. Since $z\not=w$, for some $n>1$, there are finite different subwords $A_{n},\tilde{A}_{n}$ satisfying $z\in I_{A_{n}}$ and $w\in I_{\tilde{A}_{n}}$. Therefore $I_{A_{n}}$ and $I_{\tilde{A}_{n}}$ have disjoint interiors and so do $J_{A_{n}}=h(I_{A_{n}})$ and $J_{\tilde{A}_{n}}=h(I_{\tilde{A}_{n}})$. Hence $h(z)\not=h(w)$ and $h$ is an orientation-preserving homeomorphism.

To finish the proof we need to show that the map $h\colon I\to J$ is log-singular. By hypothesis $\capacity(E)=0$ and $\capacity(h(\Sone\setminus E))=\capacity(F)=0$, where $E, F$ are given by the Lemma. Therefore $h$ is log-singular. \end{proof}

To prove that there are capacity zero sets as in Lemma \ref{lemma:logsingular} one can take the identity map in Lemma \ref{lemma:find_logsingular}. It also follows from the construction in \cite[Remark 9]{ChrisWeldingAnnals} and \cite[Proposition 2.3]{Younsi:Welding_Ex}. 

\begin{remark} It can be proved that given any log-singular map $h\colon I\to J$, the log-singular set $E$ can be taken as in Lemma \ref{lemma:logsingular} (in terms of a suitable partition $\{I_{A_{n}}\}$ of $I$). The proof follows from \cite[Lemma 11]{ChrisWeldingAnnals} and the following observation: if $h\colon I\to J$ is log-singular, then for any closed sub-arc $S\subset I$, the restriction of $h$ to $S$, $h_{S}\colon S\to h(S),$ is also log-singular. 
\end{remark}


\section{Finding regular log-singular sets. Proof of the main theorem}\label{section:finding}

In this section we prove that given a circle homeomorphism $\phi\colon\Sone\to\Sone$ there is a log-singular set $E\subset\Sone$, given by a partition as in the previous section, so that $\capacity(\phi(E))=0$. This is what we prove in Lemma \ref{lemma:find_logsingular}. Together with Lemma \ref{lemma:logsingular} it completes the proof of Theorem \ref{theorem:mainWelding2}.

\begin{lemma}[Finding regular log-singular sets]\label{lemma:find_logsingular}
Let $\phi\colon\Sone\to\Sone$ be an orientation-preserving circle homeomorphism. Then there exists partition $\{I_{A_{n}}\}$ of $\Sone$ and a log-singular set 
$$ E=\bigcap_{m\in\N}\bigcup_{n\geq m} \left(\bigcup_{A_{n}\in\mathcal{A}_{n}}\bigcup_{j=1}^{L_{n+1}/2} I_{A_{n}(2j-1)}\right),$$
so that $\capacity(\phi(E))=0$. Where the sequence $\{L_{n}\}$ associated with $\{A_{n}\}$ satisfies $L_{0}=1$ and $L_{n}$ even for $n\geq1$.
\end{lemma}
\begin{proof}

We can suppose $\phi\colon[0,1]\to[0,1]$. The goal is to decompose $[0,1]$ in an iterative way by defining a partition $\{I_{A_{n}}\}$ of $[0,1]$. However, now we need to ensure that the $I_{A_{n}(2j-2)}$ are mapped to intervals with small capacity (as opposed to Lemma \ref{lemma:logsingular}). This is represented in Figure \ref{figure:finding_logsingular_set}.

Write $I=I_{\{0\}}\cup I_{\{1\}}$, where both $I_{\{0\}}$ and $I_{\{1\}}$ are closed intervals with disjoint interiors so that $\capacity(I_{\{1\}})<2^{-1}$. Make $I_{\{1\}}$ smaller (and thus $I_{\{0\}}$ bigger) so that we also have $\capacity(\phi(I_{\{1\}}))=\capacity(J_{\{1\}})<2^{-1}$.

\begin{figure}[h]
\includegraphics[scale=1]{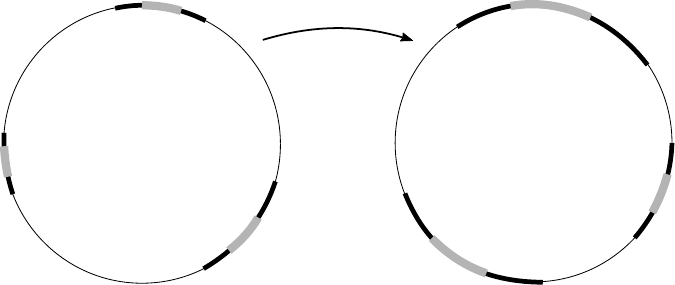}
\centering
\setlength{\unitlength}{\textwidth}
\put(-0.404,0.27){$\phi$}
\put(-0.545,0.07){$I_{\{j\}}$}
\put(-0.253,0.04){$J_{\{j\}}=\phi(I_{\{j\}})$}
\caption{Finding the partition in Lemma \ref{lemma:find_logsingular}. The lighter thicker sub-arcs have small capacity and they are mapped to sub-arcs with small capacity.}
\label{figure:finding_logsingular_set}
\end{figure}

Suppose that at the $n$-th step, for $n>1$, we have intervals $I_{A_{n}}$ and $J_{A_{n}}=h(I_{A_{n}})$, where $A_{n}=\{a_{1}, a_{2}, \ldots, a_{n}\}$, with $0\leq a_{j}< L_{j}$, is the address of $I_{A_{n}}$, so that 
$$ \capacity\left(\bigcup_{A_{n-1}\in\mathcal{A}_{n-1}}\bigcup_{j=1}^{L_{n}/2} I_{A_{n-1}(2j-1)}\right)\leq \sum_{A_{n-1}\in\mathcal{A}_{n-1}}\sum_{j=1}^{L_{n}/2} \capacity\left(I_{A_{n-1}(2j-1)}\right)< 2^{-n},$$
and their image also has capacity less than $2^{-n}$, i.e. 
$$ \capacity\left(\bigcup_{A_{n-1}\in\mathcal{A}_{n-1}}\bigcup_{j=1}^{L_{n}/2} J_{A_{n-1}(2j-1)}\right)\leq \sum_{A_{n-1}\in\mathcal{A}_{n-1}}\sum_{j=1}^{L_{n}/2} \capacity\left(J_{A_{n-1}(2j-1)}\right)< 2^{-n},$$

We build the $I_{A_{n+1}}$ in the following way: divide each $I_{A_{n}}$ first into $m\geq n$ closed segments with disjoint interior of length less than or equal than $\left(\max_{I_{A_{n}}}|I_{A_{n}}|\right)/n$ so that their images are intervals of length less than or equal than $\left(\max_{J_{A_{n}}}|J_{A_{n}}|\right)/n$. Then divide each of them into two closed segments which yields intervals $I_{A_{n+1}}$ and $J_{A_{n+1}}=h(I_{A_{n+1}})$. By making the $I_{A_{n}(2j-1)}$ smaller we obtain:
$$ \capacity\left(\bigcup_{A_{n}\in\mathcal{A}_{n}}\bigcup_{j=1}^{m L_{n}} I_{A_{n}(2j-1)}\right)\leq \sum_{A_{n}\in\mathcal{A}_{n}}\sum_{j=1}^{m L_{n}} \capacity\left(I_{A_{n}(2j-1)}\right)< 2^{-n-1},$$
and that their image also has capacity less than $2^{-n-1}$, i.e.
$$ \capacity\left(\bigcup_{A_{n}\in\mathcal{A}_{n}}\bigcup_{j=1}^{m L_{n}} J_{A_{n}(2j-1)}\right)\leq \sum_{A_{n}\in\mathcal{A}_{n}}\sum_{j=1}^{m L_{n}} \capacity\left(J_{A_{n}(2j-1)}\right)< 2^{-n-1}.$$
If we define $E$ as in the Lemma, then by sub-additivity of capacity we have,
$$ \capacity(E)+\capacity(\phi(E))\leq\sum_{n\geq m+1}(2^{-n}+2^{-n})=2^{-m}.$$
Therefore $\capacity(E)=\capacity(\phi(E))=0$.
\end{proof}

\begin{remark}
Observe that the proof of Lemma \ref{lemma:find_logsingular} can be modified to obtain a log-singular set $E$ satisfying $\capacity(E)=\capacity(\phi(E))=\capacity(\phi^{-1}(E))=0$.
\end{remark}

We finish this section by proving Theorem \ref{theorem:mainWelding} and Theorem \ref{theorem:mainWelding2}.

\begin{proof}[Proof of Theorem \ref{theorem:mainWelding} and Theorem \ref{theorem:mainWelding2}]
Given $\phi\colon\Sone\to\Sone$, then by Lemma \ref{lemma:find_logsingular} there is a log-singular set $E\subset\Sone$ so that $\capacity(\phi(E))=0$. By Lemma \ref{lemma:logsingular} there exists a log-singular circle homeomorphism $h$ satisfying $\capacity(h(\Sone\setminus E))=0$. If we consider the circle homeomorphism $\phi\circ h^{-1}$, then $h(\Sone\setminus E)$ has zero capacity and
$$ (\phi\circ h^{-1})(\Sone\setminus h(\Sone\setminus E))=\phi(E),$$
which has zero capacity. Therefore $\phi\circ h^{-1}$ is log-singular. Since log-singular maps are weldings by Theorem \ref{Theorem:ChrisLogSingular}, then 
$$ \phi=\left(\phi\circ h^{-1}\right)\circ h$$
is a composition of two conformal weldings.
\end{proof}


\bibliographystyle{amsalpha}
\bibliography{references}

\end{document}